\documentclass{amsart}

 \newtheorem{thm}{Theorem}[section]
 \newtheorem{cor}[thm]{Corollary}
 \newtheorem{lem}[thm]{Lemma}
 
 \theoremstyle{definition}
 
 \theoremstyle{remark}

 \numberwithin{equation}{section}
\newcommand{\A}{X}

\newcommand{\C}{\mathcal{C}}

\newcommand{\F}{Q}

\newcommand{\M}{\mathcal{M}}
\newcommand{\T}{\mathcal{T}}

\newcommand{\D}{D:\mathcal{T}\rightarrow \mathcal{M}}
\newcommand{\de}{d:\mathcal{T}\rightarrow \mathcal{M}}
\newcommand{\al}{\alpha:\mathcal{T}\rightarrow \mathcal{M}}
\begin{document}

\title[Jordan derivations on block upper triangular matrix algebras]
 {Jordan derivations on block upper triangular matrix algebras}

\author{ Hoger Ghahramani}

\thanks{{\scriptsize
\hskip -0.4 true cm \emph{MSC(2010)}: 16W25; 47B47; 16S50; 15B99.
\newline \emph{Keywords}: Jordan derivation; block upper triangular matrix algebra.\\}}

\address{Department of
Mathematics, University of Kurdistan, P. O. Box 416, Sanandaj,
Iran.}

\email{h.ghahramani@uok.ac.ir; hoger.ghahramani@yahoo.com}

\address{}

\email{}

\thanks{}

\thanks{}

\subjclass{}

\keywords{}

\date{}

\dedicatory{}

\commby{}


\begin{abstract}
We provide that any Jordan derivation from the block upper
triangular matrix algebra $\T = \T(n_{1},n_{2}, \cdots ,
n_{k})\subseteq M_{n}(\mathbb{\C})$ into a $2$-torsion free unital
$\T$-bimodule is the sum of a derivation and an antiderivation.
\end{abstract}

\maketitle

\section{Introduction}
Throughout this paper $\C$ will denote a commutative ring with
unity. Let $\mathcal{A}$ be an algebra over $\C$. Recall that a
$\C$-linear map $D$ from $\mathcal{A}$ into an
$\mathcal{A}$-bimodule $\M$ is said to be a \emph{Jordan
derivation} if $D(ab+ba)=D(a)b+aD(b)+D(b)a+bD(a)$ for all
$a,b\in\mathcal{A}$. It is called a \emph{derivation} if
$D(ab)=D(a)b+aD(b)$ for all $a,b\in\mathcal{A}$. If $D$ is only
additive, we will call $D$ is an \emph{additive (Jordan)
derivation}. For an element $m\in \M$, the mapping
$I_{m}:\mathcal{A}\rightarrow \M$, given by $I_{m}(a) = am-ma$,
is a derivation which will be called an \emph{inner derivation}.
Also $D$ is called an \emph{antiderivation} if
$D(ab)=D(b)a+bD(a)$ for all $a,b\in\mathcal{A}$. Clearly, each
derivation or antiderivation is a Jordan derivation. The converse
is, in general, not true (see \cite{Ben0}).
\par
The question under what conditions that a map becomes a
derivation attracted much attention of mathematicians and hence
it is natural and interesting to find some conditions under which
a Jordan derivation is a derivation. Herstein\cite{Her} proved
that every additive Jordan derivation from a $2$-torsion free
prime ring into itself is an additive derivation.
Bre$\check{\textrm{s}}$ar \cite{Bre} proved that Herstein�s
result is true for 2-torsion free semiprime rings. Sinclair
\cite{Sin} proved that every continuous Jordan derivation on
semisimple Banach algebras is a derivation. Johnson showed in
\cite{Jo} that a continuous Jordan derivation from a
$C^{*}$-algebra $\mathcal{A}$ into a Banach
$\mathcal{A}$-bimodule is a derivation. Zhang in \cite{Zh} proved
that every Jordan derivation on nest algebras is an inner
derivation. Li and Lu \cite{Li} showed that every additive Jordan
derivation on reflexive algebras is an additive derivation which
generalized the result in \cite{Zh}. By a classical result of
Jacobson and Rickart \cite{Jac} every additive Jordan derivation
on a full matrix ring over a $2$-torsion free unital ring is an
additive derivation. In \cite{Gh}, the author proved that any
additive Jordan derivation from a full matrix ring over a unital
ring into any of its $2$-torsion free bimodule (not necessarily
unital) is an additive derivation which generalized the result in
\cite{Jac}. Benkovi$\check{\textrm{c}}$ \cite{Ben0} determined
Jordan derivations on triangular matrices over commutative rings
and proved that every Jordan derivation from the algebra of all
upper triangular matrices into its arbitrary unital bimodule is
the sum of a derivation and an antiderivation. Zhang and Yu
\cite{Zh2} showed that every Jordan derivation of triangular
algebras is a derivation.
\par
In this note we prove that any Jordan derivation from the block
upper triangular matrix algebra $\T = \T(n_{1},n_{2}, \cdots ,
n_{k})\subseteq M_{n}(\mathbb{\C})$ into a $2$-torsion free unital
$\T$-bimodule is the sum of a derivation and an antiderivation,
where $\C$ is a commutative ring with unity. This result
generalizes the main result of \cite{Ben0}. Also our proof is
elementary, constructive and straightforward.
\section{Preliminaries}
Throughout this paper, by $M_{n}(\C)$, $n\geq 1$, we denote the
algebra of all $n \times n$ matrices over $\C$, by $T_{n}(\C)$ its
subalgebra of all upper triangular matrices, and by $D_{n}(\C)$
its subalgebra of all diagonal matrices. We shall denote the
identity matrix by $I$. Also, $E_{ij}$ is the matrix unit and
$x_{i,j}$ is the $(ij)$th entry of $X \in M_{n}(\C)$ for $1\leq
i,j \leq n$. Hence we have $E_{ii}XE_{jj}=x_{i,j}E_{ij}$ for
$X\in M_{n}(\C)$ and $1\leq i,j \leq n$.
\par
For $n\geq 1$ and a finite sequence of positive
integers $n_{1},n_{2}, \cdots , n_{k}$ ($k\geq 1$) , satisfying $n_{1}+n_{2}+ \cdots + n_{k}=n$, let $\T(n_{1},n_{2}, \cdots , n_{k})$ be the subalgebra of $M_{n}(\C)$ of all  matrices of the form
\[X =\begin{pmatrix}
  X_{11} & \A_{12} & \cdots & X_{1k} \\
  0 & X_{22} & \cdots & X_{2k} \\
  \vdots & \vdots & \ddots  & \vdots \\
  0 & 0 & \cdots & X_{kk}
\end{pmatrix},\]
where $X_{ij}$ is an $n_{i} \times n_{j}$ matrix. We call such an algebra a \emph{block upper triangular matrix
algebra}. Also we call $k$ is the number of \emph{summands of} $\T(n_{1},n_{2}, \cdots , n_{k})$. Note that $M_{n}(\C)$ is a special case of block upper triangular
matrix algebras. In particular, if $k=1$ with $n_{1}=n$, then $\T(n_{1},n_{2}, \cdots , n_{k})=M_{n}(\C)$. Also, when $k=n$ and $n_{i} = 1$ for
every $1\leq i \leq k$ , we have $\T(n_{1},n_{2}, \cdots , n_{k})=T_{n}(\C)$.
\par
Let $F_{1}=\sum_{i=1}^{n_{1}}E_{i}$ and
$F_{j}=\sum_{i=1}^{n_{j}}E_{i+n_{1}+\cdots+n_{j-1}}$ for $2\leq
j\leq k$, where $E_{l}=E_{ll}$. Then
$\{F_{1},\ldots,F_{k}\}$ is a set of non-trivial
idempotents of $\T(n_{1},n_{2}, \cdots , n_{k})$ such that $F_{1}+ \cdots+F_{k}=I$ and $F_{i}F_{j}=F_{j}F_{i}=0$ for $1\leq i,j \leq k$ with $i\neq j$. Moreover, we have
$F_{j}\T(n_{1},n_{2}, \cdots , n_{k})F_{j}\cong M_{n_{j}}(\C)$ for any
$1\leq j\leq k$. We use $\mathcal{D}(n_{1},n_{2}, \cdots , n_{k})$ for a subalgebra of $\T(n_{1},n_{2}, \cdots , n_{k})$ defined by
\[ \mathcal{D}(n_{1},n_{2}, \cdots , n_{k})=F_{1}\T(n_{1},n_{2}, \cdots , n_{k})F_{1}+\cdots+F_{k}\T(n_{1},n_{2}, \cdots , n_{k})F_{k}.\]
Note that, if $\T(n_{1},n_{2}, \cdots , n_{k})=T_{n}(\C)$, then $\mathcal{D}(n_{1},n_{2}, \cdots , n_{k})=D_{n}(\C)$.
\par
By $[X,Y] = XY-YX$ we denote the commutator or the Lie product of elements $X,Y\in M_{n}(\C)$.


\section{Main result}
From \cite[Theorem 3.2]{Gh} and the fact that every Jordan
derivation from $\C$ into its bimoduls is zero, we have the
following lemma which will be needed in the proofs of our results.
\begin{lem}\label{full}
Every Jordan derivation from $M_{n}(\mathcal{\C})$, for $n\geq 1$,
into any of its bimodules is a derivation.
\end{lem}
In this note, our main result is the following theorem.
\begin{thm}\label{asli}
Let $\T = \T(n_{1},n_{2}, \cdots , n_{k})$ be a block upper
triangular algebras in $M_{n}(\C)$ $(n\geq 1)$ and $\M$ be a
$2$-torsion free unital $\T$-bimodule. Suppose that $\D$ is a
Jordan derivation. Then there exist a derivation $\de$ and an
antiderivation $\al$ such that $D=d+\alpha $ and
$\alpha(\mathcal{D}(n_{1},n_{2}, \cdots , n_{k}))=\{0\}$.
Moreover, $d$ and $\alpha$ are uniquely determined.
\end{thm}
\begin{proof}
The proof is by induction on $k$, the number of summands of $\T$.
If $k=1$, then $\T=M_{n}(\C)$ and $\mathcal{D}(n_{1})=M_{n}(\C)$.
So by Lemma~\ref{full}, $D$ is derivation and $\alpha=0$ is the
only antiderivation such that $\alpha(\mathcal{D}(n_{1}))=0$.
Hence the result is obvious in this case.
\par
Assume inductively that $k\geq 1$ and the result holds for each
block upper triangular algebra $\T(n_{1},n_{2}, \cdots , n_{k})$
with $k$ summands.
\par
Let $\T = \T(n_{1},n_{2}, \cdots , n_{k+1})\subseteq M_{n}(\C)$ be
a block upper triangular algebra with $n_{1}+n_{2}+ \cdots +
n_{k+1}=n$.
\par
Set $P=F_{1}$ and $Q=I-P=F_{2}+\cdots+F_{k+1}$. Then $P$ and $Q$
are nontrivial idempotents of $\T$ such that $PQ=QP=0$. Also $Q\T
P=\{0\}$, $P\T P$ and $Q\T Q$ are subalgebras of $\T$ with unity
$P$ and $Q$, respectively, and $\T=P\T P \dot{+} P\T Q \dot{+}
Q\T Q$ as sum of $\C$-linear spaces. Moreover, $P\T P\cong
M_{n_{1}}(\C)$ and $Q\T Q\cong \T(n_{2},n_{3}, \cdots ,
n_{k+1})\subseteq M_{n-n_{1}}(\C)$ ($\C$-algebra isomorphisms) is
a block upper triangular algebra with $k$ summands
$n_{2},\cdots,n_{k+1}$, where
$\mathcal{D}(n_{2},\cdots,n_{k+1})\cong F_{2}\T
F_{2}+\cdots+F_{k+1}\T F_{k+1}$.
\par
Suppose $\M$ is a $2$-torsion free unital $\T$-bimodule and $\D$
is a Jordan derivation. Define $\Delta:\T\rightarrow \M$ by
$\Delta(X)=D(X)-I_{B}(X)$, where $B=PD(P)Q-QD(P)P$. Then $\Delta$
is a Jordan derivation such that $P\Delta(P)Q=Q\Delta(P)P=0$. We
will show that $\Delta$ is the sum of a derivation and an
antiderivation.
\par
We complete the proof by checking some steps.
\\ \\
\textbf{Step 1.}
$\Delta(X)=P\Delta(PXP)P+P\Delta(PXQ)Q+Q\Delta(PXQ)P+Q\Delta(QXQ)Q$
for all $X \in \T$.
\\ \par
Let $X \in \T$. Since $P(QXQ)+(QXQ)P=0$, we have
\begin{equation}\label{1}
P\Delta(QXQ)+\Delta(P)QXQ+QXQ\Delta(P) + \Delta(QXQ)P=0.
\end{equation}
Multiplying this identity by $P$ both on the left and on the right
we arrive at $2P \Delta(QXQ)P=0$ so $P \Delta(QXQ)P=0$. Now,
multiplying the Equation\eqref{1} from the left by $P$, from the
right by $Q$ and by the fact that $P\Delta(P)Q=0$, we find
$P\Delta(QXQ)Q=0$. Similarly, from Equation\eqref{1} and the fact
that $Q\Delta(P)P=0$, we see that $Q\Delta(QXQ)P=0$. Therefore,
from above equations we get
\begin{equation*}
\Delta(QXQ )=Q\Delta(QXQ)Q
\end{equation*}
Applying $\Delta$ to $(PXP)Q+Q(PXP)=0$, we see that
\begin{equation}\label{2}
PXP\Delta(Q)+\Delta(PXP)Q+Q\Delta(PXP)+\Delta(Q)PXP=0.
\end{equation}
By $\Delta(QXQ )=Q\Delta(QXQ)Q$, Equation\eqref{2} and using
similar methods as above we obtain
\begin{equation*}
\Delta(PXP)=P\Delta(PXP)P.
\end{equation*}
Since $P(PXQ)+(PXQ)P=PXQ$, we have
\begin{equation}\label{3}
P\Delta(PXQ)+\Delta(P)PXQ+PXQ\Delta(P)+\Delta(PXQ)P=\Delta(PXQ).
\end{equation}
Multiplying Equation\eqref{3} by $P$ both on the left and on the
right and by the fact that $Q\Delta(P)P=0$, we get
$P\Delta(PXQ)P=0$. Now multiplying Equation\eqref{3} by $Q$ both
on the left and on the right and by the fact that
$Q\Delta(P)P=0$, we have $Q\Delta(PXQ)Q=0$. Hence from these
equations we find
\begin{equation*}
\Delta(PXQ)=P\Delta(PXQ)Q+Q\Delta(PXQ)P.
\end{equation*}
Now from above results we have
\begin{equation*}
\begin{split}
\Delta(X)&=\Delta(PXP)+\Delta(PXQ)+\Delta(QXQ)\\&=P\Delta(PXP)P+P\Delta(PXQ)Q+Q\Delta(PXQ)P+Q\Delta(QXQ)Q.
\end{split}
\end{equation*}
\\
\textbf{Step 2.}
$P\Delta(PXPYP)P=PXP\Delta(PYP)P+P\Delta(PXP)PYP$ for all $X,Y
\in \T$.
\\ \par
$P\M P$ is a $2$-torsion free unital $P\T P$-bimodule. Define
$J:P\T P\rightarrow P\M P$ by $J(PXP)=P\Delta(PXP)P$. Clearly $J$
is a well defined linear map. Since $\Delta$ is a Jordan
derivation, it follows that $J$ is a Jordan derivation. By
Lemma~\ref{full} and the fact that $P\T P\cong M_{n_{1}}(\C)$, we
see that $J$ is a derivation. So we obtain the result of this
step.
\\ \\
\textbf{Step 3.} $P\Delta(PXPYQ)Q=PXP\Delta(PYQ)Q+P\Delta(PXP)PYQ$
and \\ $P\Delta(PXQYQ)\F=PXQ\Delta(QYQ)Q+P\Delta(PXQ)QYQ$ for all
$X,Y \in \T$.
\\ \par
Let $X,Y \in \T$. Applying $\Delta$ to the equations:
$PXPYQ=(PXP)(PYQ)+(PYQ)(PXP)$ and $PXQYQ=(PXQ)(QYQ)+(QYQ)(PXQ)$,
we get
\begin{equation}\label{4}
\begin{split}
&\Delta(PXPYQ)=PXP\Delta(PYQ)+\Delta(PXP)PYQ+PYQ\Delta(PXP)+\Delta(PYQ)PXP
\\& \textrm{and}
\\&
\Delta(PXQYQ)=PXQ\Delta(QYQ)+\Delta(PXQ)QYQ+QYQ\Delta(PXQ)+\Delta(QYQ)PXQ.
\end{split}
\end{equation}
Multiplying these identities by $P$ on the left and by $Q$ on the
right, from Step 1 we yield the result.
\\ \\
\textbf{Step 4.} There exists a derivation $g:Q\T Q\rightarrow
Q\M Q$ and an antiderivation $\gamma:Q\T Q\rightarrow Q\M Q$ such
that $Q\Delta(QXQ)Q=g(QXQ)+\gamma(QXQ)$ for all $X\in \T$.
Moreover, $\gamma(F_{2}\T F_{2}+\cdots+F_{k+1}\T F_{k+1})=\{0\}$
and $PXQ\gamma(QYQ)=0$ for all $X,Y\in \T$.
\\ \par
$Q\M Q$ is a $2$-torsion free unital $Q\T Q$-bimodule. Define
$G:Q\T Q\rightarrow Q\M Q$ by $G(QXQ)=Q\Delta(QXQ)Q$. Clearly $G$
is a well defined linear map. Since $\Delta$ is a Jordan
derivation, we see that $G$ is a Jordan derivation. In view of the
isomorphisms $Q\T Q\cong \T(n_{2},n_{3}, \cdots ,
n_{k+1})\subseteq M_{n-n_{1}}(\C)$,
$\mathcal{D}(n_{2},\cdots,n_{k+1})\cong F_{2}\T
F_{2}+\cdots+F_{k+1}\T F_{k+1}$ and induction hypothesis, there
exists a derivation $g:Q\T Q\rightarrow Q\M Q$ and an
antiderivation $\gamma:Q\T Q\rightarrow Q\M Q$ such that
$Q\Delta(QXQ)Q=G(QXQ)=g(QXQ)+\gamma(QXQ)$ for all $X\in \T$. Also,
$\gamma(F_{2}\T F_{2}+\cdots+F_{k+1}\T F_{k+1})=\{0\}$. We will
show that $PXQ\gamma(QYQ)=0$ for all $X,Y\in \T$.
\\
By Step 3 and above results for all $X,Y,Z\in \T$, we have
\begin{equation*}
\begin{split}
P\Delta(PXQYQZQ)Q&=PXQ\Delta(QYQZQ)Q+P\Delta(PXQ)QYQZQ\\&=PXQg(QYQZQ)+PXQ\gamma(QYQZQ)\\&+P\Delta(PXQ)QYQZQ.
\end{split}
\end{equation*}
On the other hand,
\begin{equation*}
\begin{split}
P\Delta(PXQYQZQ)Q&=PXQYQ\Delta(QZQ)Q+P\Delta(PXQYQ)QZQ\\&=PXQYQ\Delta(QZQ)Q+PXQ\Delta(QYQ)QZQ\\&+P\Delta(PXQ)QYQZQ\\&=
PXQYQg(QZQ)+PXQYQ\gamma(QZQ)\\&+PXQg(QYQ)QZQ+PXQ\gamma(QYQ)QZQ\\&+P\Delta(PXQ)QYQZQ\\&=
PXQg(QYQZQ)+PXQ\gamma(QZQYQ)\\&+P\Delta(PXQ)QYQZQ,
\end{split}
\end{equation*}
since $g$ is a derivation and $\gamma$ is an antiderivation. By
comparing the two expressions for $P\Delta(PXQYQZQ)Q$, we arrive
at
\begin{equation}\label{5}
PXQ\gamma([QYQ,QZQ])=0,
\end{equation}
for all $X,Y,Z\in \T$. Now from the fact that
$Q=F_{2}+\cdots+F_{k+1}$ and $F_{j}Q=QF_{j}=F_{j}$ for all $2\leq
j \leq k+1$, we have
\begin{equation}\label{6}
\begin{split}
QXQ-\sum_{j=2}^{k+1}F_{j}XF_{j}&=(\sum_{j=2}^{k+1}F_{j})QXQ-
\sum_{j=2}^{k+1}F_{j}XF_{j}=\sum_{j=2}^{k+1}(F_{j}XQ-F_{j}XF_{j})\\&=\sum_{j=2}^{k+1}F_{j}X(Q-F_{j})=\sum_{j=2}^{k+1}[F_{j},F_{j}X(Q-F_{j})],
\end{split}
\end{equation}
for all $X\in \T$. Note that $F_{j}, F_{j}X(Q-F_{j})\in Q\T Q$.
By Equation\eqref{5}, \eqref{6} and $\gamma(F_{2}\T
F_{2}+\cdots+F_{k+1}\T F_{k+1})=\{0\}$, we conclude that
\begin{equation*}
\begin{split}
PXQ\gamma(QYQ)&=PXQ\gamma(QYQ-\sum_{j=2}^{k+1}F_{j}YF_{j}+\sum_{j=2}^{k+1}F_{j}YF_{j})\\&=PXQ\gamma(QYQ-\sum_{j=2}^{k+1}F_{j}YF_{j})\\&
=PXQ\gamma([F_{j},F_{j}Y(Q-F_{j})])=0,
\end{split}
\end{equation*}
for all $X,Y\in \T$.
\\ \\
\textbf{Step 5.} $PXQ\Delta(PYQ)P=0$ and $Q\Delta(PXQ)PYQ=0$ for
all $X,Y\in \T$.
\\ \par
Multiplying Equations\eqref{4} by $Q$ on the left and by $P$ on
the right, we have
\begin{equation}\label{7}
\begin{split}
&Q\Delta(PXPYQ)P=Q\Delta(PYQ)PXP
\\&
\textrm{and}
\\&
Q\Delta(PXQYQ)P=QYQ\Delta(PXQ)P,
\end{split}
\end{equation}
for all $X,Y\in \T$. Now applying $\Delta$ to
$(PXQ)(PYQ)+(PYQ)(PXQ)=0$ for any $X,Y\in \T$, we see that
\[PXQ\Delta(PYQ)+\Delta(PXQ)PYQ+PYQ\Delta(PXQ)+\Delta(PYQ)PXQ=0.\]
From this identity we get the following equations.
\begin{equation}\label{8}
PXQ\Delta(PYQ)P+PYQ\Delta(PXQ)P=0
\end{equation}
and
\begin{equation}\label{9}
Q\Delta(PXQ)PYQ+Q\Delta(PYQ)PXQ=0
\end{equation}
for all $X,Y\in \T$. Let $1\leq i,k \leq n_{1}$ and $n_{1}<
j,l\leq n$ be arbitrary. By Equations\eqref{7} and
Equation\eqref{8} we have
\begin{equation*}
\begin{split}
E_{ij}\Delta(E_{kl})P&=E_{ij}\Delta(E_{ki}E_{il})P=E_{ij}\Delta(E_{il})E_{ki}\\&=E_{ij}\Delta(E_{ij}E_{jl})E_{ki}=E_{ij}E_{jl}\Delta(E_{ij})E_{ki}
=E_{il}\Delta(E_{ij})E_{ki}\\&=-E_{ij}\Delta(E_{il})E_{ki}=-E_{ij}\Delta(E_{ki}E_{il})P=-E_{ij}\Delta(E_{kl})P,
\end{split}
\end{equation*}
since $E_{ki}\in P\T P$, $E_{ij},E_{il},E_{kl}\in P\T Q$,
$E_{jl}\in Q\T Q $. So $E_{ij}\Delta(E_{kl})P=0$. Also by
Equations\eqref{7} and Equation\eqref{9} we find
\begin{equation*}
\begin{split}
Q\Delta(E_{ij})E_{kl}&=Q\Delta(E_{il}E_{lj})E_{kl}=E_{lj}\Delta(E_{il})E_{kl}\\&=E_{lj}\Delta(E_{ik}E_{kl})E_{kl}=E_{lj}\Delta(E_{kl})E_{ik}E_{kl}
=E_{lj}\Delta(E_{kl})E_{il}\\&=-E_{lj}\Delta(E_{il})E_{kl}=-Q\Delta(E_{il}E_{lj})E_{kl}=-Q\Delta(E_{ij})E_{kl},
\end{split}
\end{equation*}
since $E_{ik}\in P\T P$, $E_{ij},E_{il},E_{kl}\in P\T Q$,
$E_{lj}\in Q\T Q $. Hence $Q\Delta(E_{ij})E_{kl}=0$. For any
$X,Y\in \T$, let
$PXQ=\sum_{i=1}^{n_{1}}\sum_{j=n_{1}+1}^{n}x_{i,j}E_{ij}$ and
$PYQ=\sum_{k=1}^{n_{1}}\sum_{l=n_{1}+1}^{n}y_{k,l}E_{kl}$.
Therefore, by identities $E_{ij}\Delta(E_{kl})P=0$,
$Q\Delta(E_{ij})E_{kl}=0$ and linearity of $\Delta$ it follows
that
\begin{equation*}
\begin{split}
PXQ\Delta(PYQ)P&=\sum_{i=1}^{n_{1}}\sum_{j=n_{1}+1}^{n}x_{i,j}E_{ij}\Delta(\sum_{k=1}^{n_{1}}\sum_{l=n_{1}+1}^{n}y_{k,l}E_{kl})P\\&=
\sum_{i=1}^{n_{1}}\sum_{j=n_{1}+1}^{n}\sum_{k=1}^{n_{1}}\sum_{l=n_{1}+1}^{n}x_{i,j}y_{k,l}E_{ij}\Delta(E_{kl})P=0,
\end{split}
\end{equation*}
and
\begin{equation*}
\begin{split}
Q\Delta(PXQ)PYQ&=\sum_{k=1}^{n_{1}}\sum_{l=n_{1}+1}^{n}Q\Delta(\sum_{i=1}^{n_{1}}\sum_{j=n_{1}+1}^{n}x_{i,j}E_{ij})y_{k,l}E_{kl}\\&=
\sum_{k=1}^{n_{1}}\sum_{l=n_{1}+1}^{n}\sum_{i=1}^{n_{1}}\sum_{j=n_{1}+1}^{n}x_{i,j}y_{k,l}Q\Delta(E_{ij})E_{kl}=0.
\end{split}
\end{equation*}
\\ \\
\textbf{Step 6.} The mapping $\delta:\T\rightarrow \M$, given by
$$\delta(X)=P\Delta(PXP)P+P\Delta(PXQ)Q+g(QXQ)$$ is a derivation
and the mapping $\alpha:\T\rightarrow \M$, given by
$$\alpha(X)=Q\Delta(PXQ)P+\gamma(QXQ)$$ is an antiderivation such
that $\alpha(\mathcal{D}(n_{1},n_{2}, \cdots , n_{k+1}))=\{0\}$.
Moreover, $\Delta=\delta+\alpha$.
\\ \par
Clearly, $\delta$ is a linear map. By Steps 2, 3, 4 and the fact
that $Q\T P=\{0\}$ one can check directly that $\delta$ is a
derivation.
\\
It is clear that $\alpha$ is a linear map. For each $X,Y\in \T$,
by Equations\eqref{7}, Steps 4, 5 and the fact that $Q\T P=\{0\}$,
we have
\begin{equation*}
\begin{split}
\alpha(XY)&=Q\Delta(PXPYQ)P+Q\Delta(PXQYQ)P+\gamma(QXQYQ)\\&=Q\Delta(PYQ)PXP+QYQ\Delta(PXQ)P+QYQ\gamma(QXQ)+\gamma(QYQ)QXQ\\&+
Q\Delta(PYQ)PXQ+PYQ\Delta(PXQ)P+PYQ\gamma(QXQ)+\gamma(QYQ)QXP\\&=
Y\alpha(X)+\alpha(Y)X.
\end{split}
\end{equation*}
Let
$F_{1}X_{1}F_{1}+F_{2}X_{2}F_{2}+\cdots+F_{k+1}X_{k+1}F_{k+1}$ be
an arbitrary element of $\mathcal{D}(n_{1},n_{2}, \cdots ,
n_{k+1})$. Since $\gamma(F_{2}\T F_{2}+\cdots+F_{k+1}\T
F_{k+1})=\{0\}$, $F_{1}Q=QF_{1}=0$, $PF_{j}=F_{j}P=0$ and
$F_{j}Q=QF_{j}=F_{j}$ for any $2\leq j \leq k$, it follows that
\begin{equation*}
\begin{split}
&\alpha(F_{1}X_{1}F_{1}+F_{2}X_{2}F_{2}+\cdots+F_{k+1}X_{k+1}F_{k+1})\\&=
Q\Delta(P(F_{1}X_{1}F_{1}+F_{2}X_{2}F_{2}+\cdots+F_{k+1}X_{k+1}F_{k+1})Q)P\\&
+\gamma(Q(F_{1}X_{1}F_{1}+F_{2}X_{2}F_{2}+\cdots+F_{k+1}X_{k+1}F_{k+1})Q)\\&=
\gamma(F_{2}X_{2}F_{2}+\cdots+F_{k+1}X_{k+1}F_{k+1})=0.
\end{split}
\end{equation*}
So $\alpha(\mathcal{D}(n_{1},n_{2}, \cdots , n_{k+1}))=\{0\}$. By
Steps 1, 4, it is obvious that $\Delta=\delta+\alpha$.
\\ \par
Now from the above results we have
$D-I_{B}=\Delta=\delta+\alpha$, where $\delta:\T\rightarrow \M $
is a derivation, $\alpha:\T\rightarrow \M $ is an antiderivation
and $\alpha(\mathcal{D}(n_{1},n_{2}, \cdots , n_{k+1}))=\{0\}$. So
the mapping $d:\T\rightarrow \M$ given by $d=\delta+I_{B}$ is a
derivation and we find $D=d+\alpha$.
\par
Finally, we will show that $d$ and $\alpha$ are uniquely
determined. Suppose that $D=d^{\prime}+\alpha^{\prime}$, where
$d^{\prime}:\T\rightarrow \M $ is a derivation,
$\alpha^{\prime}:\T\rightarrow \M $ is an antiderivation and
$\alpha^{\prime}(\mathcal{D}(n_{1},n_{2}, \cdots ,
n_{k+1}))=\{0\}$. Hence $D_{|Q\T Q}:Q\T Q\rightarrow \M$, the
restriction of $D$ to $Q\T Q$, is a Jordan derivation. So
$D_{|Q\T Q}=d_{|Q\T Q}+\alpha_{|Q\T Q}=d^{\prime}_{|Q\T
Q}+\alpha^{\prime}_{|Q\T Q}$, where $d_{|Q\T Q},d^{\prime}_{|Q\T
Q}:Q\T Q\rightarrow \M $ are derivations, $\alpha_{|Q\T Q},
\alpha^{\prime}_{|Q\T Q}:Q\T Q\rightarrow \M $ are antiderivations
and $\alpha_{|Q\T Q}, \alpha^{\prime}_{|Q\T Q}(F_{2}\T
F_{2}+\cdots+F_{k+1}\T F_{k+1})=\{0\}$. Since $Q\T Q\cong
\T(n_{2},n_{3}, \cdots , n_{k+1})\subseteq M_{n-n_{1}}(\C)$,
$\mathcal{D}(n_{2},\cdots,n_{k+1})\cong F_{2}\T
F_{2}+\cdots+F_{k+1}\T F_{k+1}$, by the uniqueness in induction
hypothesis it follows that $\alpha_{|Q\T Q}=\alpha^{\prime}_{|Q\T
Q}$ and $d_{|Q\T Q}=d^{\prime}_{|Q\T Q}$. Define
$\beta:\T\rightarrow \M$ by $\beta=\alpha-\alpha^{\prime}$.
Clearly $\beta$ is a linear map and
$\beta=\alpha-\alpha^{\prime}=d^{\prime}-d$. So $\beta$ is a
derivation and an antiderivation. Since
$\alpha(\mathcal{D}(n_{1},n_{2}, \cdots ,
n_{k+1}))=\alpha^{\prime}(\mathcal{D}(n_{1},n_{2}, \cdots ,
n_{k+1}))=\{0\}$, it follows that
$\alpha(PXP)=\alpha^{\prime}(PXP)=0$ for all $X\in \T$, so
$\beta(PXP)=0$ for all $X\in \T$. Also from $\alpha_{|Q\T
Q}=\alpha^{\prime}_{|Q\T Q}$, we have $\beta(QXQ)=0$ for all
$X\in \T$. Now observe that $\beta(P)=\beta(Q)=0$. Then, since
$\beta$ is a derivation and an antiderivation, we have
\begin{equation*}
\begin{split}
         \beta(PXQ)& =P\beta(XQ)+\beta(P)XQ=P\beta(XQ)\\ & =
         P(Q\beta(X)+\beta(Q)X)=0.
\end{split}
\end{equation*}
 So
$$\beta(X)=\beta(PXP)+\beta(PXQ)+\beta(QXQ)=0$$
for all $X\in \T$. Therefore, $\alpha=\alpha^{\prime}$ and hence
$d=d^{\prime}$. The proof of Theorem~\ref{asli} is thus completed.
\end{proof}
We have the following corollary, which was proved in \cite{Ben0}.
\begin{cor}
Let $T_{n}(\C)$ be an upper triangular matrix algebra and $\M$ be
a $2$-torsion free unital $T_{n}(\C)$-bimodule. Suppose that
$D:T_{n}(\C)\rightarrow \M$ is a Jordan derivation. Then there
exist a derivation $d:T_{n}(\C)\rightarrow \M$ and an
antiderivation $\alpha:T_{n}(\C)\rightarrow \M$ such that
$D=d+\alpha $ and $\alpha(D_{n}(\C))=\{0\}$. Moreover, $d$ and
$\alpha$ are uniquely determined.
\end{cor}
\subsection*{Acknowledgment}
\bibliographystyle{amsplain}
\bibliography{xbib}

\end{document}